\documentclass[11pt]{amsart}
\usepackage{amsmath, amssymb, tensor, amsthm, amsfonts, marginnote, tikz-cd, mathtools, stmaryrd, enumitem}
\usepackage[mathscr]{euscript}
\usepackage[hidelinks]{hyperref}
\setlist[enumerate]{itemsep=0.25em}
\usepackage[skip=4pt, indent=10pt]{parskip}

\newcommand\cA{\mathscr{A}}
\newcommand\cB{\mathscr{B}}
\newcommand\cC{\mathscr{C}}
\newcommand\cD{\mathscr{D}}
\newcommand\cE{\mathscr{E}}
\newcommand\cF{\mathscr{F}}
\newcommand\cG{\mathscr{G}}
\newcommand\cH{\mathscr{H}}

\newcommand\cM{\mathscr{M}}

\newcommand\cO{\mathscr{O}}

\def\bZ{\mathbf{Z}}

\def\sfB{\mathsf{B}}
\def\sfC{\mathsf{C}}
\def\sfG{\mathsf{G}}

\def\sfS{\mathsf{S}}
\def\sfT{\mathsf{T}}
\def\sfX{\mathsf{X}}
\def\sfY{\mathsf{Y}}
\def\sfZ{\mathsf{Z}}

\newcommand{\beq}{\begin{equation}}
\newcommand{\eeq}{\end{equation}}

\theoremstyle{plain}
\newtheorem{theorem}[subsection]{Theorem}
\newtheorem{proposition}[subsection]{Proposition}
\newtheorem{lemma}[subsection]{Lemma}

\theoremstyle{definition}
\newtheorem*{definition}{Definition}
\newtheorem{remark}[subsection]{Remark}

\DeclareMathOperator{\QCoh}{QCoh}
\DeclareMathOperator{\Coh}{Coh}
\DeclareMathOperator{\IndCoh}{IndCoh}
\DeclareMathOperator{\TwoIndCoh}{2IndCoh}
\DeclareMathOperator{\Perf}{Perf}
\DeclareMathOperator{\IndPerf}{IndPerf}
\DeclareMathOperator{\St}{\mathsf{St}}
\DeclareMathOperator{\pt}{pt}
\DeclareMathOperator{\module}{-mod}
\DeclareMathOperator{\Hom}{Hom}
\DeclareMathOperator{\fnt}{fnt}\DeclareMathOperator{\cpt}{cpt}
\DeclareMathOperator{\psc}{psc}
\DeclareMathOperator{\cone}{cone}
\DeclareMathOperator{\Fl}{Fl}
\DeclareMathOperator{\id}{id}
\DeclareMathOperator{\colim}{colim}

\DeclareMathOperator{\act}{act}
\DeclareMathOperator{\Shv}{Shv}
\DeclareMathOperator{\Nilp}{Nilp}
\DeclareMathOperator{\End}{End}
\DeclareMathOperator{\Vect}{Vect}
\DeclareMathOperator{\sct}{sc}
\DeclareMathOperator{\rest}{rest}
\DeclareMathOperator{\Spec}{Spec}
\DeclareMathOperator{\ad}{ad}

\title{Tame local Betti Langlands}
\author{Gurbir Dhillon and Jeremy Taylor}

\begin{document}
\begin{abstract}
Working over the integers, we prove that the universal monodromic affine Hecke category is monoidally equivalent to ind-coherent sheaves on the Steinberg stack, confirming a conjecture of Ben-Zvi and Nadler. Specializing to unipotent monodromy, this provides another argument for a fundamental theorem of Bezrukavnikov.
\end{abstract}
\maketitle

\setcounter{tocdepth}{1} 
\tableofcontents

\section{Introduction}
In this paper we establish tame local Betti geometric Langlands, confirming a conjecture of Ben-Zvi and Nadler \cite{BZN07, BZFN12, BZN18} and extending Bezrukavnikov's equivalence \cite{Bez16, BR24}. We work with integral coefficients, which introduces some additional homological complications compared to the arguments in \cite{Bez16}.

\subsection*{Background}
Let $G$ be a complex reductive group, and $\sfG$ be the Langlands dual group.

The constructible side of tame geometric Langlands concerns sheaves on the moduli stack of $G$-bundles on a curve with tame ramification at a finite subset of points.

This depends on a choice of one of the following sheaf theories:
\begin{enumerate}
\item[-] Betti means all sheaves in the analytic topology with nilpotent singular support,
\item[-] de Rham means all algebraic D-modules,
\item[-] restricted means ind-constructible sheaves with nilpotent singular support.
\end{enumerate}

Each of these sheaf theories gives rise to a different moduli space of tame $\sfG$-local systems on the punctured curve. 
The Betti and de Rham moduli spaces are analytically but not algebraically isomorphic, and restricted moduli space is a certain common formal substack.

Now we specialize to the tame local case, or equivalently a genus zero curve with two points of tame ramification. Then both sides of geometric Langlands admit natural convolution monoidal structures. They act by tame Hecke operators on the categories associated to punctured curves.

If we require unipotent monodromy around the two punctures, then all three versions of the moduli space of $\sfG$-local systems become isomorphic, and the unipotent monodromic equivalence was proved by Bezrukavnikov \cite{Bez16} for characteristic zero field coefficients. This was extended to modular field coefficients (except for some small primes) by Bezrukavnikov--Riche--Rider and Bezukavnikov--Riche \cite{BRR20, BR22, BR24}. They were partly motivated by the Finkelberg--Mirković conjecture \cite{FM99}.

By contrast, allowing the monodromy to vary produces three genuinely different moduli spaces of local systems. Our main theorem extends Bezrukavnikov's equivalence to Betti families and integer coefficients. Namely we will prove a family of equivalences over dual torus whose fiber at the identity recovers Bezrukavnikov's equivalence.

\subsection*{Notation} Before stating our main theorem, we need the following notation.

Let $I \subset G[[t]]$ be the prounipotent radical of the Iwahori subgroup, i.e. the preimage of a maximal unipotent $N \subset G$ under the evaluation map $G[[t]] \rightarrow G$.

Consider the universal monodromic affine Hecke category \[\cH \coloneqq \Shv_{\Nilp}(I \backslash G((t))/I).\]
This is the category of sheaves in the analytic topology with integer coefficients and nilpotent singular support.

Explicitly, the nilpotent singular support condition means that we only consider sheaves that are locally constant (with arbitrary monodromy) along the orbits of the maximal torus.

Following the sheaf theoretic conventions of \cite{RS18}, we allow unbounded complexes and impose no finiteness conditions on stalks. For example the monoidal unit in $\cH$ has infinite dimensional stalks. It is the local system on the unit orbit that corresponds to the regular representation of the fundamental group.

Let $\sfG$ be the Langlands dual group defined over the integers. Let $\sfB \subset \sfG$ be the opposite Borel subgroup. The multiplicative Grothendieck--Springer scheme is defined by $\widetilde{\sfG} \coloneqq \sfG \times^\sfB \sfB$. The Steinberg scheme is the fiber product $\St \coloneqq \widetilde{\sfG} \times_\sfG \widetilde{\sfG}$.

\subsection*{Statement of results}

Our main theorem is the following equivalence between the universal monodromic affine Hecke category and ind-coherent sheaves on the Steinberg stack, compatible with the convolution monoidal structures.

\begin{theorem}\label{Main}
There is a monoidal equivalence
\[\cH \simeq \IndCoh_\sfG(\St).\]
\end{theorem}

Under this equivalence, Gaitsgory's central sheaves \cite{Gai01} correspond to vector bundles supported on the diagonal. Moreover the Iwahori--Whittaker averaging functors correspond to the direct image functors along the projections to $\widetilde{\sfG}/\sfG$.

Passing to unipotent monodromic subcategories on both sides gives an integral coefficient enhancement of Bezrukavnikov's equivalence \cite{Bez16, BR24}. 

\subsection*{Construction of the functor}
Now we will outline the proof of Theorem \ref{Main}.
It is based on the following idea that is due to Kazhdan--Lusztig in the function theoretic setting \cite{KL87} and Bezrukavnikov in the unipotent monodromic setting \cite{Bez16}. Namely they match the action of the affine Hecke category on the Iwahori--Whittaker module with the action of the Steinberg stack on the Grothendieck--Springer scheme.

\begin{remark}
We would like to realize both sides of Theorem \ref{Main} as endomorphisms of the Whittaker module that commute with the bi-Whittaker action. However this recovers only the tempered quotients. Most of this paper concerns how to obtain Theorem \ref{Main} from its tempered version using an analysis of the t-structures.
\end{remark}

Following \cite{KL87, AB09, Bez16}, we first need a coherent realization of the suitably defined Iwahori--Whittaker module $\cH_\chi$. This is provided in \cite{DT25} by the universal monodromic Arkhipov--Bezrukavnikov equivalence \[\cH_\chi \simeq \IndPerf_\sfG(\widetilde{\sfG}).\]
We then consider the convolution actions \[\cH \curvearrowright \cH_\chi \quad \text{ and } \quad \IndCoh_\sfG(\St) \curvearrowright \IndPerf_\sfG(\widetilde{\sfG}).\]

\begin{remark} In the function theoretic setting of \cite{KL87}, these modules are faithful. Therefore it suffices to identify the images of the affine Hecke algebra and the equivariant K-homology of the Steinberg variety inside the common algebra of endomorphisms.
\end{remark}

In our categorical setting, the resulting functor \[\IndCoh_\sfG(\St) \rightarrow \End(\IndPerf_\sfG(\widetilde{\sfG}))\]
is far from fully faithful. Indeed, even ignoring renormalization issues, direct image of coherent sheaves along $\St/\sfG \rightarrow \widetilde{\sfG}/\sfG \times \widetilde{\sfG}/\sfG$ is not fully faithful.

For this reason, we will also consider the commuting action of the bi-Whittaker category ${}_\chi \cH_\chi$. We already proved in \cite{DT25} that it admits a coherent realization, namely there is a monoidal equivalence \[{}_\chi \cH_\chi \simeq \IndPerf_\sfG(\sfG).\]

Thus we will obtain a monoidal functor
\[\iota^!: \cH \rightarrow \End_{{}_\chi \cH_\chi}(\cH_\chi) \simeq \IndPerf_\sfG(\widetilde{\sfG}) \underset{\IndPerf_\sfG(\sfG)}{\otimes} \IndPerf_\sfG(\widetilde{\sfG}),\]
that admits a fully faithful left adjoint on general grounds.

\subsection*{Renormalization}
This functor $\iota^!$ is not yet an equivalence for reasons related to the infinite dimensionality of the affine flag variety and the singularities of the Steinberg stack. For example the dualizing sheaf on the affine flag variety belongs to the kernel, because it acts trivially on the Iwahori--Whittaker module. To obtain an equivalence, we will analyze t-structures as follows.

First, by adapting a beautiful idea from Lemma 38 of \cite{Bez16}, we will prove that the kernel of $\iota^!$ is contained inside the infinitely connective subcategory of the affine Hecke category. This formally implies that the restriction of $\iota^!$ to the eventually coconnective subcategory is fully faithful.

\begin{remark}
In our integral coefficient setting, not all infinitely connective objects are killed by $\iota^!$. By contrast, for characteristic zero field coefficients, the kernel of $\iota^!$ is the entire infinitely connective subcategory, and $\IndPerf_{\sfG}(\St)$ is the left completion of the affine Hecke category.
\end{remark}

To proceed we will assume that $G$ has connected center. Then we will construct a t-structure on $\IndPerf_\sfG(\widetilde{\sfG}) \otimes_{\IndPerf_\sfG(\sfG)} \IndPerf_\sfG(\widetilde{\sfG})$ such that the eventually coconnective subcategory is equivalent to that of quasi-coherent sheaves on the Steinberg stack. For this we will use a monadicity argument that will require us to first establish certain special cases of the categorical Kunneth equivalence for ind-perfect complexes.

Then we will prove that $\iota^!$ is t-bounded with respect to this t-structure. 
This formally implies that it induces a t-bounded equivalence $\cH^{> -\infty} \simeq \QCoh_\sfG(\St)^{> -\infty}$ on eventually coconnective subcategories.

Finally we will pass to `pseudo-compact' objects on both sides to obtain the desired equivalence between compact objects in the affine Hecke category and coherent sheaves on the Steinberg stack.

\subsection*{Comparison to other work}
As already mentioned, our arguments are based on \cite{Bez16}.

However Bezrukavnikov's construction is somewhat different from ours, because he does not consider the bi-Whittaker category.
Rather he constructs a functor that is automatically linear for the spectral action of perfect complexes on the Steinberg variety, equipped with its pointwise tensor product. (This is similar to the construction in \cite{GR24} of the global geometric Langlands functor.) To prove that his functor is monoidal for the convolution product requires an additional argument, performed in Section 10 of \cite{Bez16}, that involves the Whittaker module and tilting sheaves.

Bezrukavnikov writes:
\begin{quote}
In this text we follow the original plan conceived more than a decade ago and treat the issues of homological algebra by ad hoc methods, using explicit DG models for triangulated categories of constructible sheaves based on generalized tilting sheaves. While the properties of tilting sheaves established in the course of the argument are (in the author’s opinion) of an independent interest, it is likely that recent advances in homotopy algebra can be used to develop an alternative approach.
\end{quote}
In this paper, we use exactly such advances in homotopy algebra, together with the bi-Whittaker equivalence, to construct a functor $\iota^!$ that is automatically a monoidal colocalization.

Our construction of the functor is also influenced by unpublished approaches of Ben-Zvi--Nadler, Gaitsgory--Lurie and Arinkin--Bezrukavnikov to reprove derived Satake \cite{BF07} using the action of the spherical Hecke category on the spherical Whittaker category. The approach of Gaitsgory--Lurie has been realized in \cite{CR23}.

After constructing the functor, our approach to renormalization is based on two cohomological estimates, Propositions \ref{BoundedCompact} and \ref{TBounded}, that were both influenced by \cite{Bez16}.

Our proof of fully faithfulness is based on Proposition \ref{BoundedCompact}. This uses the idea from Lemma 38 of \cite{Bez16} of convolving by a very dominant costandard sheaf.

To prove essential surjectivity, Bezrukavnikov's uses a cohomological estimate and a general argument about approximating coherent sheaves by perfect complexes. These are Propositions 36 and 35(b) of \cite{Bez16} respectively. However his proof of the former cohomological estimate does not seem to work in our setting, because $\Perf_\sfG(\widetilde{\sfG})$ has infinite cohomological dimension over the integers.

For this reason we need a somewhat different argument for essential surjectivity. Instead of Bezrukavnikov's cohomological estimate, we will use a weaker t-boundedness result. Namely we use Proposition \ref{TBounded}, which is similar to Proposition 41 of \cite{Bez16}, but we give a softer proof that avoids producing an explicit bound. Roughly we will argue that only finitely many coherent sheaves (indexed by the finite Weyl group) are responsible for the non t-exactness of $\iota^!$.

\subsection*{Related results} Here we discuss some variants of Theorem \ref{Main}.

By formally completing at various points in the dual torus, Theorem \ref{Main} implies Conjecture 58 of \cite{Bez16} in the complex setting. Indeed, both sides of Theorem \ref{Main} admit a symmetric monoidal functor from $\QCoh(\sfT/\!/W^{\fnt})$. Let $\QCoh(\sfT/\!/ W^{\fnt})^{\rest} \subset \QCoh(\sfT/\!/ W^{\fnt})$ be the full subcategory generated by skyscraper sheaves. Applying the base change
\[- \otimes_{\QCoh(\sfT/\!/W^{\fnt})} \QCoh(\sfT/\!/W^{\fnt})^{\rest}\]
to both sides of Theorem \ref{Main} gives the restricted tame local geometric Langlands equivalence. This provides an alternative to the endoscopic arguments of \cite{DLYZ25}.

By passing to modules, Theorem \ref{Main} implies the following equivalence of 2-categories. Recall that $\widetilde{\sfG}/\sfG \rightarrow \sfG/\sfG$ induces a Lagrangian correspondence between cotangent bundles. Transporting the zero section in $T^*(\widetilde{\sfG}/\sfG)$ across this correspondence gives exactly the nilpotent subset of $T^*(\sfG/\sfG)$. (For this reason, all coherent sheaves on the Steinberg stack automatically have nilpotent singular support.) Moreover $\widetilde{\sfG}/\sfG \rightarrow \sfG/\sfG$ is a proper map between smooth stacks. Therefore the 2-category $\TwoIndCoh_{\Nilp}(\sfG/\sfG)$ of ind-coherent sheaves of categories on $\sfG/\sfG$ with nilpotent singular support (as developed by Arinkin) is equivalent to modules for $\IndCoh_\sfG(\St)$. Hence Theorem \ref{Main} implies an equivalence of 2-categories \[\cH \module \simeq \TwoIndCoh_{\Nilp}(\sfG/\sfG).\]

The statement of Theorem \ref{Main} was used in \cite{LNY24} to calculate functions on the commuting stack. It has various other expected variants, such as in the mixed characteristic \cite{AGLR22, ALWY23} and de Rham settings.

\subsection*{Acknowledgments} It is a pleasure to thank Pramod Achar, David Ben-Zvi, Jens Eberhardt, Arnaud Eteve, Joakim F\ae rgeman, Dennis Gaitsgory, Xinyu Li, Yau Wing Li, Ivan Losev, David Nadler, Jiahao Niu, Sam Raskin, Simon Riche, David Yang, Zhiwei Yun, and Xinwen Zhu for useful correspondence and discussions. We especially thank Harrison Chen, with whom many of the basic ideas of the present paper were jointly formulated, and Roman Bezrukavnikov, for several conversations which greatly influenced our understanding of this subject. ChatGPT was used for checking and proofreading. 

Part of this research was performed while the author was visiting SLMath, which is supported by the NSF grant DMS-2424139. G.D. was supported by an NSF grant MSPRF 2103387. J.T. was supported by NSF grants DMS-1646385 and MSPRF 2503562.

\section{Preliminaries}
Here we will collect some preliminaries. We will recall some generalities about coherent sheaves on algebraic stacks. Then we will review the affine Hecke category and prove that its Iwahori--Whittaker module is properly dualizable.

\subsection*{Coherent sheaves on algebraic stacks}
The theory of coherent sheaves on algebraic stacks is more complicated over the integers (compared to over a field of characteristic zero) because perfect complexes typically have infinite cohomological dimension.

Throughout this paper, we will assume that all algebraic stacks are finite type over the integers and eventually coconnective.

The category of quasi-coherent sheaves on a stack is left complete with respect to its natural t-structure. However it is typically not compactly generated \cite{HR15, HNR19}.

The category of ind-coherent sheaves on an algebraic stack is defined following \cite{Zhu25} as the ind-completion of the category of bounded coherent complexes. (This is different from \cite{GR19} where ind-coherent sheaves are defined using descent starting from affine schemes.) It inherits a natural t-structure because coherent sheaves are closed under truncation. Moreover the functor to quasi-coherent sheaves induces an equivalence on eventually coconnective subcategories.

A quasi-coherent sheaf is called perfect if it is dualizable with respect to the symmetric monoidal structure. This is equivalent to the condition that its pullback to any scheme admits a finite resolution by vector bundles. 

Perfect complexes on an eventually  coconnective algebraic stack are coherent. Therefore we get a fully faithful functor from ind-perfect complexes to ind-coherent sheaves, which is an equivalence for smooth algebraic stacks.

\subsection*{The affine Hecke category}
Here we collect some notation regarding the affine Hecke category.

Let $\Fl \coloneqq G((t))/I$ be the enhanced affine flag variety. Let $T \subset G$ be a maximal torus. Let $\Lambda$ be the coweight lattice, and $\Lambda^+$ be the submonoid of dominant coweights. Let $W^{\fnt}$ be the finite Weyl group. Let $\sfT \coloneqq \Spec \bZ[\Lambda]$ be the dual torus defined over the integers.

The Iwahori-orbits in $\Fl$ are indexed by the affine Weyl group $W \coloneqq W^{\fnt} \ltimes \Lambda$. Each such orbit is a $T$-torsor over an affine vector space.
Therefore the category of locally constant sheaves on each orbit is non-canonically equivalent to $\QCoh(\sfT)$.

The universal local system (i.e. regular representation of the fundamental group of $T$) corresponds to the structure sheaf of $\sfT$.
Let $\Delta_w$ and $\nabla_w \in \cH$ be the standard and costandard extension of this universal local system on the Iwahori orbit indexed by $w \in W$. They are shifted to be perverse.

\subsection*{The Iwahori--Whittaker category}
Here we review the Iwahori--Whittaker category.

Let $\cH^{\fnt} \coloneqq \Shv_{\Nilp}(N \backslash G/N)$ denote the universal monodromic finite Hecke category. Let $\Xi \in \cH^{\fnt}$ be the universal monodromic big tilting sheaf.

Write $\sfG = \sfG^{\sct}/\sfZ$, where $\sfG^{\sct}$ has simply connected derived subgroup and $\sfZ$ is a finite central subgroup. Set $\sfC \coloneqq (\sfT^{\sct}/\!/W^{\fnt})/\sfZ$, where $\sfT^{\sct} \subset \sfG^{\sct}$ is a maximal torus.

Recall from \cite{IY25, LNY24, Tay25} that there is a continuous monoidal functor
\[\Hom(\Xi, -): \cH^{\fnt} \rightarrow \QCoh(\sfT \times_\sfC \sfT)\]
that admits a fully faithful left adjoint.

Composing this functor with the convolution action gives left and right actions of $\cH^{\fnt}$ on $\QCoh(\sfT)$. Using these actions, define the Iwahori--Whittaker modules \[\cH_\chi \coloneqq \cH \otimes_{\cH^{\fnt}} \QCoh(\sfT) \quad \text{ and } \quad {}_\chi\cH \coloneqq \QCoh(\sfT) \otimes_{\cH^{\fnt}} \cH\]
and the bi-Whittaker category \[{}_\chi\cH_\chi \coloneqq \QCoh(\sfT) \otimes_{\cH^{\fnt}} \cH \otimes_{\cH^{\fnt}} \QCoh(\sfT).\]

\subsection*{Proper dualizability}
The following lemma says that the Iwahori--Whittaker category is `properly dualizable' in the sense of \cite{BZGO20}.

\begin{lemma}\label{ProperDualizable}
The module categories $\cH \curvearrowright \cH_\chi$ and ${}_\chi\cH \curvearrowleft \cH$ are dual to each other. Moreover
\begin{enumerate}[label=(\alph*)]
\item\label{ProperDualizable1} the unit $u:{}_{\chi}\cH_{\chi} \xrightarrow{\sim} {}_{\chi}\cH \otimes_{\cH} \cH_{\chi}$ is an equivalence,
\item \label{ProperDualizable2} the counit $c: \cH_{\chi} \otimes_{{}_{\chi}\cH_{\chi}} {}_{\chi}\cH \rightarrow \cH$ is a $\cH$-bimodule functor that preserves compactness.
\end{enumerate}
\end{lemma}
\begin{proof}
The modules $\QCoh(\sfT \times_\sfC \sfT) \curvearrowright \QCoh(\sfT)$ and $\QCoh(\sfT) \curvearrowleft \QCoh(\sfT \times_\sfC \sfT)$ are dual to each other, and the counit preserves compactness. Lemma \ref{DualColocalization} implies that these categories are also dual to each other as $\cH^{\fnt}$-modules, and the counit
\[\QCoh(\sfT) \otimes \QCoh(\sfT) \rightarrow \cH^{\fnt}, \qquad \cO_\sfT \boxtimes \cO_\sfT \mapsto \Xi\]
preserves compactness. Since the affine Hecke category is rigid, the modules $\cH \curvearrowright \cH_\chi$ and ${}_\chi\cH \curvearrowleft \cH$ are dual to each other, such that the unit satisfies \ref{ProperDualizable1} and the counit satisfies \ref{ProperDualizable2}.
\end{proof}

We used the following general result to show that the counit preserves compactness.

\begin{lemma}\label{DualColocalization}
Let $R: \cA \rightarrow \cB$ be a continuous monoidal functor, and $\cB \curvearrowright \cM$ be a module.
Assume that
\begin{enumerate}[label=(\alph*)]
\item\label{DualColocalization1} $R$ admits a fully faithful left adjoint $L$,
\item\label{DualColocalization2} $\cM$ is dualizable as a $\cB$-module, and the counit preserves compactness,
\item\label{DualColocalization3} $\cA$ is rigid.
\end{enumerate}
Then $\cM$ is also dualizable as an $\cA$-module, and the counit preserves compactness.
\end{lemma}
\begin{proof}
Denote the unit and counit of $\cB$-module duality by \[u_\cB: \Vect \rightarrow \cM^\vee \otimes_\cB \cM \quad \text{ and } \quad c_\cB: \cM \otimes \cM^\vee \rightarrow \cB.\]
Assumption \ref{DualColocalization1} implies that $R$ induces an equivalence $\cM^\vee \otimes_\cA \cM \simeq \cM^\vee \otimes_\cB \cM$. Moreover \ref{DualColocalization3} implies that $L$ is $\cA$-bilinear.

Commutativity of
\[\begin{tikzcd}[column sep = large]
\cM \arrow[r, "\id \otimes u_\cB"] & \cM \otimes \cM^\vee \otimes_\cB \cM \arrow[r, "c_\cB \otimes \id"] & \cB \otimes_\cB \cM \arrow[r, "\sim"] & \cM \\
& \cM \otimes \cM^\vee \otimes_\cA \cM \arrow[u, "\sim"] \arrow[r, "c_\cB \otimes \id"'] & \cB \otimes_\cA \cM \arrow[u, "\sim"] \arrow[r, "L \otimes \id"'] & \cA \otimes_\cA \cM \arrow[u, "\sim"']
\end{tikzcd}\]
and a similar diagram for the other triangle identity, imply that
\[u_\cA: \Vect \xrightarrow{u_\cB} \cM^\vee \otimes_\cB \cM \simeq \cM^\vee \otimes_\cA \cM \quad \text{ and } \quad c_\cA: \cM \otimes \cM^\vee \xrightarrow{c_\cB} \cB \xrightarrow{L} \cA\]
satisfy the triangle identities. Therefore $\cM$ is dualizable as an $\cA$-module.

Since $R$ is continuous, $L$ preserves compactness. Therefore \ref{DualColocalization2} implies that $c_\cA$ also preserves compactness as desired.
\end{proof}

\section{Construction of the functor}
Here we will use the action of the affine Hecke category on the Iwahori--Whittaker module to construct a monoidal colocalization \[\iota^!: \cH \rightarrow \IndPerf_\sfG(\widetilde{\sfG}) \underset{\IndPerf_\sfG(\sfG)}{\otimes} \IndPerf_\sfG(\widetilde{\sfG}).\]

\subsection*{Bi-Whittaker and Arkhipov--Bezukavnikov equivalences}
Here we recall the main results of \cite{DT25}.

First we have the following equivalence between ind-perfect complexes on the adjoint quotient and the universal monodromic bi-Whittaker category. This extends results of \cite{Bez09, CD23} to integral coefficients and universal monodromy.

\begin{theorem}\label{BiWhit}
There is a monoidal equivalence \[\IndPerf_{\sfG}(\sfG) \simeq {}_\chi \cH_\chi.\]
\end{theorem}

We also have the following equivalence between ind-perfect complexes on the Grothendieck--Springer stack and Iwahori--Whittaker sheaves on the affine flag variety. This extends the Arkhipov--Bezrukavnikov equivalence \cite{AB09, AR25} to integer coefficients and universal monodromy.

Write $\cO(\lambda)$ for the $\sfG$-equivariant line bundle on $\widetilde{\sfG}$ pulled back from a character $\lambda \in \Lambda$ of $\sfB$. Then $\cO(\lambda)$ is ample if and only if $\lambda \in \Lambda^{++}$ is strictly dominant (because $\sfB$ is the opposite Borel).

\begin{theorem}\label{AB}
There is a monoidal functor \beq\label{ABFunctor} \IndPerf_{\sfG}(\widetilde{\sfG}) \rightarrow \cH,\eeq
satisfying the following properties.
\begin{enumerate}[label=(\alph*)]
\item\label{AB1} It sends $\cO(\lambda) \mapsto \nabla_{\lambda}$, for $\lambda \in \Lambda^+$ dominant.
\item\label{AB2} Composing it with Whittaker averaging gives an equivalence \[\IndPerf_{\sfG}(\widetilde{\sfG}) \simeq \cH_\chi,\]
that is equivariant for the action of the bi-Whittaker category.
\end{enumerate}
\end{theorem}

\subsection*{Construction of the functor}
Here we construct a functor from the affine Hecke category to a tensor product of categories of ind-perfect complexes. Moreover we show that it admits a fully faithful left adjoint by a general result of \cite{BZGO20}.

\begin{proposition}\label{TheFunctor}
There is a $\cH$-bilinear monoidal functor
\[\iota^!: \cH \rightarrow \IndPerf_\sfG(\widetilde{\sfG}) \underset{\IndPerf_\sfG(\sfG)}{\otimes} \IndPerf_\sfG(\widetilde{\sfG}),\]
that admits a $\cH$-bilinear fully faithful left adjoint $\iota_!$.
\end{proposition}
\begin{proof}
Using the commuting actions of $\cH$ and ${}_\chi \cH_\chi$ on the Iwahori--Whittaker bimodule $\cH_\chi$, we get a monoidal functor \begin{align*}\iota^!: \cH &\rightarrow \End_{{}_\chi \cH_\chi}(\cH_\chi) \\
&\simeq \cH_\chi \otimes_{{}_\chi \cH_\chi} {}_\chi \cH & \text{Lemma } \ref{ProperDualizable} \\
&\simeq \IndPerf_{\sfG}(\widetilde{\sfG}) \otimes_{\IndPerf_{\sfG}(\sfG)} \IndPerf_{\sfG}(\widetilde{\sfG}). & \text{Theorems } \ref{BiWhit} \text{ and } \ref{AB}\ref{AB2}
\end{align*}
Lemmas \ref{ProperDualizable} and \ref{BZGO} imply that it admits a fully faithful left adjoint. By rigidity of $\cH$, this left adjoint is $\cH$-bilinear.
\end{proof}

\begin{remark}
The fully faithful left adjoint $\iota_!$ is not essentially surjective, rather its essential image is the two-sided ideal generated by the big tilting sheaf.
\end{remark}

\begin{remark}
By analyzing the t-structures, we will ultimately prove that $\iota^!$ induces an equivalence between compact objects in the affine Hecke category and a certain full subcategory of $\IndPerf_\sfG(\widetilde{\sfG}) \otimes_{\IndPerf_\sfG(\sfG)} \IndPerf_\sfG(\widetilde{\sfG})$ that is equivalent to $\Coh_\sfG(\St)$. It is not immediate that coherent sheaves naturally form a full subcategory of the tensor product category, but this will follow from Proposition \ref{Monad}.
\end{remark}

\subsection*{Dualizability}
To show that the left adjoint was fully faithful, we used the following general result of \cite{BZGO20}. We will present an alternative diagrammatic proof based on the argument from Section 4.2.5 of \cite{BLR24}.

\begin{lemma}\label{BZGO}
Let $\cA$ and $\cB$ be rigid monoidal categories. Let $\cM$ be a bimodule category for $\cA$ and $\cB$. Assume that $\cM$ is dualizable and satisfies the following conditions.
\begin{enumerate}[label=(\alph*)]
\item\label{BZGO1} The unit and counit \[u: \cB \rightarrow \cM^{\vee} \otimes_{\cA} \cM \quad \text{ and } \quad c: \cM \otimes_{\cB} \cM^{\vee} \rightarrow \cA\]
both preserve compactness.
\item\label{BZGO2} The right adjoint to the unit is fully faithful.
\end{enumerate}
Then the counit is fully faithful.
\end{lemma}
\begin{proof}
Since the unit and counit preserve compactness, they admit continuous right adjoints
\[u^R: \cM^{\vee} \otimes_{\cA} \cM \rightarrow \cB  \quad \text{ and } \quad c^R: \cA \rightarrow \cM \otimes_{\cB} \cM^{\vee}\]
that themselves satisfy the triangle identities. Moreover rigidity implies that $u^R$ is $\cB$-bilinear and $c^R$ is $\cA$-bilinear.

The triangle identities and $uu^R \simeq \id$ give \[c^Rc \simeq (c \otimes c^Rc)(\id \otimes u \otimes \id) \simeq (c \otimes \id \otimes \id)(\id \otimes uu^R \otimes \id)(\id \otimes \id \otimes c^Rc)(\id \otimes u \otimes \id) \simeq \id\]
as depicted in the following diagram.
\begin{center} \includegraphics[width=9.5cm]{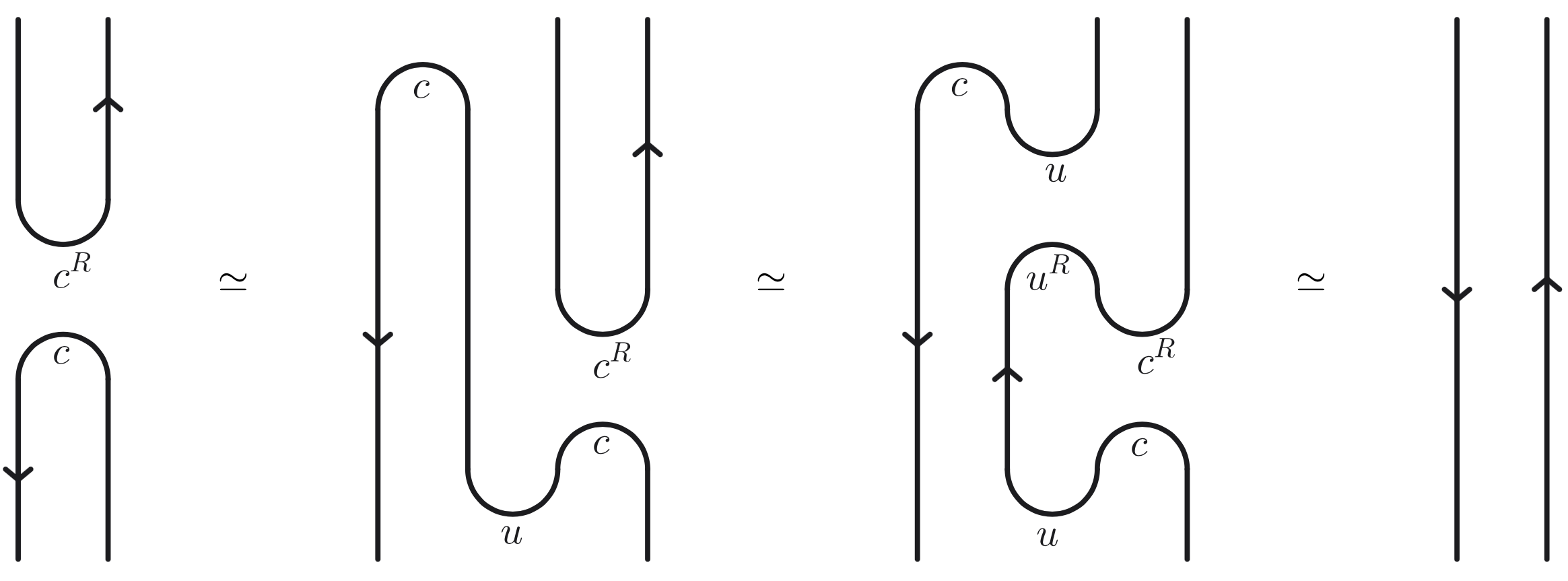} \end{center}
Since $c^Rc$ is an equivalence, $c$ is fully faithful.
\end{proof}

\section{Fully faithfulness}
Here we will prove that the kernel of $\iota^!$ is contained inside the infinitely connective subcategory of the affine Hecke category. We will formally deduce that $\iota^!$ is fully faithful on the eventually coconnective subcategory.

\subsection*{Notation}
Here we collect some notation regarding t-structures.

Let $\cC$ be a category equipped with a t-structure.
\begin{enumerate}
\item[-] Write $\cC^{\leq -\infty} \coloneqq \cap \cC^{\leq -n}$ for the subcategory of infinitely connective objects.
\item[-] Write $\cC^{> -\infty} \coloneqq \cup \cC^{\geq -n}$ for the subcategory of eventually coconnective objects.
\end{enumerate}
If $A \in \cC$, we write `$A \leq n$' (respectively `$A \geq n$') as shorthand to mean that $A$ is concentrated in degrees $\leq n$ (respectively $\geq n$).

Recall that the perverse t-structure on $\cH$ is defined by
\begin{enumerate}
\item[-] $A \geq 0$ if and only if $\Hom(\Delta_w, A) \geq 0$ for all $w \in W$,
\item[-] $A \leq 0$ if and only if $\Hom(A, \nabla_w \otimes_R M) \geq 0$ for all $M$ in the abelian category of $R$-modules and $w \in W$.
\end{enumerate}

\subsection*{The kernel is infinitely connective}
In the function theoretic setting of \cite{KL87}, Lemma 5.3 of \cite{Lus89} implies that the anti-spherical module is a faithful representation of the affine Hecke algebra.

The following proposition provides a categorification of this result. Its proof is based on the following idea from Lemma 38 of \cite{Bez16}.
Namely given a compact object in the affine Hecke category, we will convolve it by a sufficiently dominant costandard sheaf, so that all costalks of the resulting sheaf vanish except those indexed by elements of the affine Weyl group in the dominant chamber.

Let $\Lambda^{++} \subset \Lambda$ denote the subset of strictly dominant coweights.

\begin{proposition}\label{BoundedCompact}
There is a containment $\ker \iota^! \subset \cH^{\leq -\infty}$.
\end{proposition}
\begin{proof}
We will first consider the case of compact objects. Assume that $A \in \cH^{\cpt}$ is compact and we will prove that
\begin{enumerate}
\item[($\star$)] $\iota^! A \simeq 0$ implies $A \simeq 0$.
\end{enumerate}
Choose $\mu \in \Lambda^+$ sufficiently dominant so that
\beq \label{StarDominant} A * \nabla_{\mu} \text{ is a finite colimit of } \nabla_v \text{ indexed by } v \in W^{\fnt}\Lambda^{++} \subset W.\eeq
Also assume that $\iota^! A \simeq 0$. Then Lemma \ref{CharacterizeKernel} gives \beq\label{ConvolveNablaA} A * \nabla_\mu * \Xi \simeq 0.\eeq
We will now prove that $A * \nabla_\mu \simeq 0$ by homing into it from standards.

\begin{enumerate}[label=(\roman*)]
\item[-] If $w \in W^{\fnt}\Lambda^{++}$, then
\[\Hom(\Delta_w, A * \nabla_\mu) \overset{\eqref{StarDominant}, \eqref{HomXiHomDelta}}{\simeq} \Hom(\Delta_w, A * \nabla_\mu * \Xi) \overset{\eqref{ConvolveNablaA}}{\simeq} 0.\]

\item[-] If $w \not\in W^{\fnt}\Lambda^{++}$, then \[\Hom(\Delta_w, A * \nabla_\mu) \overset{\eqref{StarDominant}}{\simeq} 0.\]
\end{enumerate}
Therefore $A * \nabla_\mu \simeq 0$, and hence also \[A \simeq A * \nabla_\mu * \Delta_{-\mu} \simeq 0.\] This completes the proof of $(\star)$ when $A$ is compact.

Now we consider the case of not necessarily compact objects. Using that the affine Hecke category is generated by convolution exact sheaves, we will be able to reduce this to the case of compact objects that was already resolved above.

Let $A \in \ker \iota^!$, not necessarily compact. Let $\Xi_w \in \cH$ be an indecomposable tilting sheaf whose support is the closure of the Iwahori orbit indexed by $w \in W$. Recall that tilting sheaves are convolution exact. Therefore, for all $w \in W$ and integers $i$, we have \beq\label{TiltingVanishing} H^i(A) * \Xi_w * \Xi \simeq H^i(A * \Xi_w * \Xi) \simeq 0.\eeq

Let $B \subset H^i(A)$ be a (compact) perverse subsheaf. Equation \eqref{TiltingVanishing} and convolution exactness of tilting sheaves imply that $B * \Xi_w * \Xi \simeq 0$. Since $\Xi_w$ generate the affine Hecke category, Lemma \ref{CharacterizeKernel} gives $\iota^! B \simeq 0$. Hence ($\star$) for compact objects gives $B \simeq 0$.

Every nonzero perverse sheaf in $\cH$ admits a nonzero compact perverse subsheaf. Therefore all cohomology sheaves $H^i(A) \simeq 0$ vanish, i.e. $A$ is infinitely connective.
\end{proof}

In the above argument, we used the following lemmas.

\begin{lemma}
If $w, v \in W^{\fnt}\Lambda^{++}$ then \beq \label{HomXiHomDelta}\Hom(\Delta_w, \nabla_v) \simeq \Hom(\Delta_w, \nabla_v * \Xi).\eeq
\end{lemma}
\begin{proof}
This follows because
\begin{enumerate}
\item[-] $\nabla_v * \Xi$ admits a filtration with graded pieces $\nabla_{vz}$ for $z \in W^{\fnt}$,
\item[-] $w$ and $v$ are both minimal length in their respective left $W^{\fnt}$-cosets. \qedhere
\end{enumerate}
\end{proof}

We also used the following characterization of the kernel of $\iota^!$.

\begin{lemma}\label{CharacterizeKernel}
An object $A \in \cH$ is killed by $\iota^!$ if and only if $A * B * \Xi \simeq 0$ for all $B \in \cH$.
\end{lemma}
\begin{proof}
The left adjoint to the Iwahori--Whittaker averaging functor $\cH \rightarrow \cH_\chi$ is conservative, and the corresponding comonad is $- * \Xi$. Therefore both conditions are equivalent to $A$ acting trivially on $\cH_\chi$.
\end{proof}

\subsection*{Fully faithfulness}
Proposition \ref{BoundedCompact} formally implies that $\iota^!$ is fully faithful on the eventually coconnective subcategory by the following general result.

\begin{proposition}\label{GeneralFullyFaithful}
Let $\iota^!: \cC \rightarrow \cD$ be a functor that admits a fully faithful left adjoint $\iota_!$. Suppose that $\cC$ admits a t-structure such that $\ker \iota^! \subset \cC^{\leq -\infty}$.

Then after restricting to the eventually coconnective category,
\[\iota^!: \cC^{> -\infty} \rightarrow \cD\]
is fully faithful.
\end{proposition}
\begin{proof}
Let $A, B \in \cC$. Fully faithfulness of $\iota_!$ implies that $\iota^! \cone(\iota_! \iota^! A \rightarrow A) \simeq 0$. Therefore $\cone(\iota_! \iota^! A \rightarrow A) \leq -\infty$ is infinitely connective.

If $B > -\infty$ is eventually coconnective, then
\[\Hom(\cone(\iota_! \iota^! A \rightarrow A), B) \simeq 0\]
and hence $\Hom(A, B) \simeq \Hom(\iota^!A, \iota^! B)$.
Therefore $\iota^!$ is fully faithful on $\cC^{> -\infty}$.
\end{proof}

\section{The categorical Kunneth functor}
Here we will prove some generalities about the categorical Kunneth functor for ind-perfect complexes on algebraic stacks. Namely we will prove that it is fully faithful under general hypothesis and that it is essentially surjective in certain special cases.

Recall that we are assuming all algebraic stacks are finite type over the integers and eventually coconnective.

\subsection*{Fully faithfulness}
The following lemma gives fully faithfulness of the categorical Kunneth functor.

\begin{lemma}\label{KunnethFullyFaithful}
Let $\sfS$ be a smooth algebraic stack. Let $\sfX$ and $\sfY$ be algebraic stacks schematic over $\sfS$. Write $d: \sfX \times_\sfS \sfY \rightarrow \sfX \times \sfY$ for the partial diagonal.

Then the categorical Kunneth functor \[\IndPerf(\sfX) \otimes_{\IndPerf(\sfS)} \IndPerf(\sfY) \rightarrow \IndPerf(\sfX \times_\sfS \sfY), \qquad \cE \boxtimes \cF \mapsto \cE \boxtimes_\sfS \cF \coloneqq d^*(\cE \boxtimes \cF)\]
is fully faithful.
\end{lemma}
\begin{proof}
Let $\cE, \cE' \in \Perf(\sfX)$ and $\cF, \cF' \in \Perf(\sfY)$. Since perfect complexes are dualizable, we have
\beq \begin{aligned}\label{PerfectHom}\underline{\Hom}(\cE, \cE') \boxtimes_\sfS \underline{\Hom}(\cF, \cF') &\simeq (\cE^{\vee} \otimes \cE') \boxtimes_\sfS (\cF^{\vee} \otimes \cF')\\
&\simeq (\cE \boxtimes_\sfS \cF)^{\vee} \otimes (\cE' \boxtimes_\sfS \cF') \\
&\simeq \underline{\Hom}(\cE \boxtimes_\sfS \cF, \cE' \boxtimes_\sfS \cF').\end{aligned} \eeq
Therefore, using the base change isomorphism from Lemma \ref{IndPerfDirectImage} applied to the Cartesian square
\beq\label{CartesianGraph} \begin{tikzcd} \sfX \times_\sfS \sfY \arrow[r, "d"] \arrow[d] & \sfX \times \sfY \arrow[d, "c"] \\
\sfS \arrow[r, "d"'] & \sfS \times \sfS,\end{tikzcd}\eeq
we get
\begin{align*}\Hom_{\IndPerf(\sfX) \underset{\IndPerf(\sfS)}{\otimes} \IndPerf(\sfY)}(\cE \boxtimes \cF, \cE' \boxtimes \cF')
&\simeq \Gamma(\sfS, d^*c_*(\underline{\Hom}(\cE, \cE') \boxtimes \underline{\Hom}(\cF, \cF'))) \\
&\overset{\eqref{CartesianGraph}, \eqref{BaseChange}}{\simeq} \Gamma(\sfX \times_\sfS \sfY, \underline{\Hom}(\cE, \cE') \boxtimes_\sfS \underline{\Hom}(\cF, \cF')) \\
&\overset{\eqref{PerfectHom}}{\simeq} \Hom_{\Perf(\sfX \times_\sfS \sfY)}(\cE \boxtimes_\sfS \cF, \cE' \boxtimes_\sfS \cF'). \end{align*}
Since $\IndPerf(\sfX) \otimes_{\IndPerf(\sfS)} \IndPerf(\sfY)$ is compactly generated by such $\cE \boxtimes \cF$, the categorical Kunneth functor is fully faithful.
\end{proof}

\begin{remark}
In general we need to use the derived fiber product in order to get base change for the Cartesian square \eqref{CartesianGraph}. However derived algebraic geometry is not necessary for our intended application, because the Steinberg stack happens to remain classical.
\end{remark}

We used the following lemma to get base change for ind-perfect complexes.

\begin{lemma}\label{IndPerfDirectImage}
Let $f: \sfX \rightarrow \sfS$ be a schematic map of algebraic stacks. The functor $f^*: \IndPerf(\sfS) \rightarrow \IndPerf(\sfX)$ admits a continuous right adjoint $f_*:\IndPerf(\sfX) \rightarrow \IndPerf(\sfS)$.

For every Cartesian diagram
\[\begin{tikzcd}
\sfX \arrow[r, "g"] \arrow[d, "f"'] & \sfY \arrow[d, "f"] \\
\sfS \arrow[r, "g"'] & \sfT
\end{tikzcd}\]
such that $\sfT$ and $\sfS$ are smooth, the base change natural transformation \beq\label{BaseChange} g^*f_* \rightarrow f_*g^*\eeq is an isomorphism of functors $\IndPerf(\sfY) \rightarrow \IndPerf(\sfS)$.
\end{lemma}
\begin{proof}
By Lemma 9.21 of \cite{Zhu25} and smoothness, there is a t-structure on $\IndPerf(\sfS)$ whose eventually coconnective subcategory agrees with that of quasi-coherent sheaves. The right adjoint to $f^*: \IndPerf(\sfS) \rightarrow \IndPerf(\sfX)$ is the ind-extension of the functor \[\Perf(\sfX) \rightarrow \QCoh(\sfS)^{> -\infty} \simeq \IndPerf(\sfS)^{> -\infty} \rightarrow \IndPerf(\sfS),\]
where $\Perf(\sfX) \rightarrow \QCoh(\sfS)^{> -\infty}$ is direct image of quasi-coherent sheaves.

Therefore base change for quasi-coherent sheaves gives the desired base change for ind-perfect complexes.
\end{proof}

\subsection*{Essential surjectivity}
The following lemma says that the categorical Kunneth functor is essentially surjective whenever the diagonal can be equivariantly resolved by exterior tensor products.

\begin{lemma}\label{DiagonalResolution}
Let $\sfS$ be a smooth algebraic stack. Let $\sfX$ and $\sfY$ be algebraic stacks schematic over $\sfS$. Assume that
\begin{enumerate}[label=(\alph*)]
\item\label{DiagonalResolution1} the diagonal $\Delta_* \cO_\sfX \in \Perf(\sfX \times_\sfS \sfX)$ can be written as a finite colimit of exterior tensor products $\cF \boxtimes_\sfS \cG$ for $\cF, \cG \in \Perf(\sfX)$,
\item\label{DiagonalResolution2} both $\sfY$ and $\sfX \times_{\sfS} \sfY$ are smooth.
\end{enumerate}

Then the categorical Kunneth functor \beq\label{IndPerfKunneth} \IndPerf(\sfX) \otimes_{\IndPerf(\sfS)} \IndPerf(\sfY) \rightarrow \IndPerf(\sfX \times_\sfS \sfY)\eeq
is an equivalence.
\end{lemma}
\begin{proof}
Let \[p_{13}, p_{23}: \sfX \times_\sfS \sfX \times_\sfS \sfY \rightarrow \sfX \times_\sfS \sfY, \quad p_{12}: \sfX \times_\sfS \sfX \times_\sfS \sfY \rightarrow \sfX \times_\sfS \sfX, \quad p: \sfX \times_\sfS \sfY \rightarrow \sfY\]
be the projections.

Let $\cE \in \Perf(\sfX \times_\sfS \sfY)$, and we will show that it can be written as a finite colimit of exterior tensor products. Indeed \ref{DiagonalResolution1} implies that \[\cE \simeq p_{13*}((p_{12}^*\Delta_* \cO_\sfX) \otimes (p_{23}^* \cE))\]
can be written as a finite colimit of the quasi-coherent sheaves \begin{align*}p_{13*}((p_{12}^*(\cF \boxtimes_\sfS \cG)) \otimes (p_{23}^* \cE)) &\simeq p_{13*}((p_{13}^*(\cF \boxtimes_\sfS \cO_\sfY)) \otimes (p_{23}^* (\cE \otimes (\cG \boxtimes_\sfS \cO_\sfY)))) \\
&\simeq (\cF \boxtimes_\sfS \cO_\sfY) \otimes (p_{13*}p_{23}^* (\cE \otimes (\cG \boxtimes_\sfS \cO_\sfY))) \\
&\simeq \cF \boxtimes_\sfS p_*(\cE \otimes (\cG \boxtimes_\sfS \cO_\sfY))) \end{align*}
for $\cF, \cG \in \Perf(\sfX)$.
Therefore $\cE$ can be written as a finite colimit in $\QCoh(\sfX \times_\sfS \sfY)^{> -\infty}$ of exterior tensor products $\cF \boxtimes_\sfS \cH$ for $\cF \in \Perf(\sfX)$ and $\cH \in \QCoh(\sfY)^{> -\infty}$.

Lemma 9.21 of \cite{Zhu25} and the smoothness assumption \ref{DiagonalResolution2} give t-structures such that there are equivalences $\IndPerf(\sfY)^{> -\infty} \simeq \QCoh(\sfY)^{> -\infty}$ and $\IndPerf(\sfX \times_\sfS \sfY)^{> -\infty} \simeq \QCoh(\sfX \times_\sfS \sfY)^{> -\infty}$. Moreover $\IndPerf(\sfX \times_\sfS \sfY)^{> -\infty} \rightarrow \IndPerf(\sfX \times_\sfS \sfY)$ preserves finite colimits. 
Hence $\cE$ can be written as a finite colimit in $\IndPerf(\sfX \times_\sfS \sfY)$ of exterior tensor products $\cF \boxtimes_\sfS \cH$ for $\cF \in \Perf(\sfX)$ and $\cH \in \IndPerf(\sfY)$.

Therefore \eqref{IndPerfKunneth} is essentially surjective.  Lemma \ref{KunnethFullyFaithful} says that it is also fully faithful.
\end{proof}

\begin{remark}
If the fundamental group of $\sfG$ has torsion, then the categorical Kunneth functor for the self intersection of the diagonal map $\pt/\sfG \rightarrow \pt/(\sfG \times \sfG)$ is not essentially surjective by Proposition VIII.5.11 of \cite{FS21}.
\end{remark}

\section{Construction of a t-structure}
Here we will construct a t-structure on $\IndPerf_\sfG(\widetilde{\sfG}) \otimes_{\IndPerf_\sfG(\sfG)} \IndPerf_\sfG(\widetilde{\sfG})$, by showing that the functor to $\IndPerf_\sfG(\widetilde{\sfG} \times \widetilde{\sfG})$ is conservative and using a monadicity argument.

\subsection*{Resolution of the diagonal}
The following lemma implies that the hypotheses of Lemma \ref{DiagonalResolution} are satisfied whenever one of the factors is $\widetilde{\sfG}/\sfG \rightarrow \pt/\sfG$.

\begin{lemma}\label{PerfInd}
Assume that $G$ has connected center. Then the diagonal $\Delta_* \cO_{\widetilde{\sfG}} \in \Perf_\sfG(\widetilde{\sfG} \times \widetilde{\sfG})$ can be written as a finite colimit of equivariant line bundles $\cO(\lambda) \boxtimes \cO(\mu)$.
\end{lemma}
\begin{proof}
There is a Cartesian square
\[\begin{tikzcd}
\sfB/\sfB \arrow[r, "\Delta"]\arrow[d] & \sfB/\sfB \times_{\pt/\sfB} \sfB/\sfB \arrow[d] \\
1/\sfB \arrow[r, "\Delta"'] & \sfB/\sfB.
\end{tikzcd}\]
The proof of Proposition VIII.5.11 of \cite{FS21} shows that $\Delta_* \cO_{1/\sfB}$ is a finite colimit of line bundles on $\sfB/\sfB$ that are pulled back from representations of $\sfB$. Therefore by base change, $\Delta_* \cO_{\sfB/\sfB}$ is a finite colimit of line bundles on $\sfB/\sfB \times_{\pt/\sfB} \sfB/\sfB$ that are pulled back from representations of $\sfB \times \sfB$. 

There is a Cartesian square
\[\begin{tikzcd}
\sfB/\sfB \times_{\pt/\sfB} \sfB/\sfB \arrow[r, "\Delta"] \arrow[d] & \sfB/\sfB \times_{\pt/\sfG} \sfB/\sfB \arrow[d] \\
\pt/\sfB \times_{\pt/\sfB} \pt/\sfB \arrow[r, "\Delta"'] & \pt/\sfB \times_{\pt/\sfG} \pt/\sfB.
\end{tikzcd}\]
Using that $G$ has connected center, Proposition VIII.5.11 of \cite{FS21} (see also Lemma 16.4.2 of \cite{DT25}) shows that $\Delta_* \cO_{\pt/\sfB \times_{\pt/\sfB} \pt/\sfB}$ is a finite colimit of line bundles on $\pt/\sfB \times_{\pt/\sfG} \pt/\sfB = \sfB \backslash \sfG/\sfB$ that are pulled back from representations of $\sfB \times \sfB$. Therefore by base change, $\Delta_* \cO_{\sfB/\sfB \times_{\pt/\sfB} \sfB/\sfB}$ is a finite colimit of line bundles on $\sfB/\sfB \times_{\pt/\sfG} \sfB/\sfB$ that are pulled back from representations of $\sfB \times \sfB$.
\end{proof}

\subsection*{Kunneth functor for the Grothendieck--Springer stack}
The following proposition says that the categorical Kunneth functor is an equivalence whenever one of the factors is $\widetilde{\sfG}/\sfG \rightarrow \pt/\sfG$. This will be used to calculate a certain monad in the proof of Proposition \ref{Monad}.

\begin{proposition}\label{GrothSpringKunneth}
Assume that $G$ has connected center. Let $\sfY$ be a smooth scheme acted on by $\sfG$. Then the categorical Kunneth functor is an equivalence \beq\label{EqKunneth} \IndPerf_\sfG(\widetilde{\sfG}) \otimes_{\IndPerf_\sfG(\pt)} \IndPerf_\sfG(\sfY) \simeq \IndPerf_\sfG(\widetilde{\sfG} \times \sfY).\eeq
\end{proposition}
\begin{proof}
Condition \ref{DiagonalResolution1} of Lemma \ref{DiagonalResolution} holds by Lemma \ref{PerfInd}. Therefore Lemma \ref{DiagonalResolution} gives the desired Kunneth equivalence.
\end{proof}

\subsection*{Construction of a t-structure}
Here we will construct a t-structure on $\IndPerf_\sfG(\widetilde{\sfG}) \otimes_{\IndPerf_\sfG(\sfG)} \IndPerf_\sfG(\widetilde{\sfG})$, such that the eventually coconnective subcategory agrees with that of quasi-coherent sheaves on the Steinberg stack. 

The category $\IndPerf_{\sfG}(\widetilde{\sfG} \times \widetilde{\sfG})$ admits a natural t-structure because $\widetilde{\sfG} \times \widetilde{\sfG}$ is smooth. To construct the t-structure on $\IndPerf_\sfG(\widetilde{\sfG}) \otimes_{\IndPerf_\sfG(\sfG)} \IndPerf_\sfG(\widetilde{\sfG})$, we will interpret it as the category of modules over a certain algebra object in $\IndPerf_{\sfG}(\widetilde{\sfG} \times \widetilde{\sfG})$.

Write $d: \St \rightarrow \widetilde{\sfG} \times \widetilde{\sfG}$ for the partial diagonal.

\begin{proposition}\label{Monad}
Assume that $G$ has connected center. There exists a unique t-structure on $\IndPerf_\sfG(\widetilde{\sfG}) \otimes_{\IndPerf_\sfG(\sfG)} \IndPerf_\sfG(\widetilde{\sfG})$ satisfying the following properties.
\begin{enumerate}[label=(\alph*)]
\item\label{Monad1} The functor
\begin{align*}\delta_*: \IndPerf_\sfG(\widetilde{\sfG}) \underset{\IndPerf_\sfG(\sfG)}{\otimes} \IndPerf_\sfG(\widetilde{\sfG}) &\xrightarrow{\eqref{RightTensor}} \IndPerf_\sfG(\widetilde{\sfG}) \underset{\IndPerf_\sfG(\pt)}{\otimes} \IndPerf_\sfG(\widetilde{\sfG})\\
&\overset{\eqref{EqKunneth}}{\simeq} \IndPerf_{\sfG}(\widetilde{\sfG} \times \widetilde{\sfG})\end{align*}
is t-exact.
\item\label{Monad2} There is a t-exact equivalence on eventually coconnective subcategories
\[(\IndPerf_\sfG(\widetilde{\sfG}) \underset{\IndPerf_\sfG(\sfG)}{\otimes} \IndPerf_\sfG(\widetilde{\sfG}))^{> -\infty} \simeq \QCoh_{\sfG}(\St)^{> -\infty}.\]
\end{enumerate}
\end{proposition}
\begin{proof}
First we prove \ref{Monad1}.
Let $\Gamma$ be the graph of the Grothendieck--Springer map $\widetilde{\sfG} \rightarrow \sfG$. Lemma \ref{TensorMonad} and Proposition \ref{GrothSpringKunneth} imply that $\delta_*$ is monadic. Moreover Proposition \ref{GrothSpringKunneth} implies that the resulting monad is given by
\[\IndPerf_\sfG(\widetilde{\sfG} \times \widetilde{\sfG}) \xrightarrow{(\id \times \Gamma)_*} \IndPerf_\sfG(\widetilde{\sfG} \times \sfG \times \widetilde{\sfG}) \xrightarrow{(\Gamma \times \id)^*} \IndPerf_\sfG(\widetilde{\sfG} \times \widetilde{\sfG}).\]
By base change along
\[\begin{tikzcd}
\St \arrow[r, "d"] \arrow[d, "d"'] & \widetilde{\sfG} \times \widetilde{\sfG} \arrow[d, "\id \times \Gamma"]\\
\widetilde{\sfG} \times \widetilde{\sfG} \arrow[r, "\Gamma \times \id"'] & \widetilde{\sfG} \times \sfG \times \widetilde{\sfG},
\end{tikzcd}\]
this monad is also given by \[\IndPerf_\sfG(\widetilde{\sfG} \times \widetilde{\sfG}) \xrightarrow{d^*} \IndPerf_\sfG(\St) \xrightarrow{d_*} \IndPerf_\sfG(\widetilde{\sfG} \times \widetilde{\sfG}).\]

Hence there is an equivalence \[\IndPerf_\sfG(\widetilde{\sfG}) \underset{\IndPerf_\sfG(\sfG)}{\otimes} \IndPerf_\sfG(\widetilde{\sfG}) \simeq \cA\module(\IndPerf_{\sfG}(\widetilde{\sfG} \times \widetilde{\sfG}))\]
where $\cA \coloneqq d_* \cO_{\St}$.
Since $\cA$ is concentrated in degrees $\leq 0$ (in fact degree zero), we get a unique t-structure satisfying \ref{Monad1}.

Now we prove \ref{Monad2}. Since the partial diagonal morphism is affine, the functor
\[d_* : \QCoh_\sfG(\St) \rightarrow \QCoh_\sfG(\widetilde{\sfG} \times \widetilde{\sfG})\]
is monadic.
Therefore we get equivalences
\beq\label{AModSquare} \begin{tikzcd}
\IndPerf_\sfG(\widetilde{\sfG}) \underset{\IndPerf_\sfG(\sfG)}{\otimes} \IndPerf_\sfG(\widetilde{\sfG}) \arrow[r] \arrow[d, "\sim"'] & \QCoh_{\sfG}(\St) \arrow[d, "\sim"] \\
\cA\module(\IndPerf_{\sfG}(\widetilde{\sfG} \times \widetilde{\sfG})) \arrow[r] & \cA\module(\QCoh_{\sfG}(\widetilde{\sfG} \times \widetilde{\sfG})).\end{tikzcd}\eeq
Lemma 9.21 of \cite{Zhu25} and smoothness of $\widetilde{\sfG} \times \widetilde{\sfG}$ gives an equivalence \[\IndPerf_{\sfG}(\widetilde{\sfG} \times \widetilde{\sfG})^{> -\infty} \simeq \QCoh_{\sfG}(\widetilde{\sfG} \times \widetilde{\sfG})^{> -\infty}.\]
Therefore restricting \eqref{AModSquare} to eventually coconnective subcategories gives the desired \ref{Monad2}.
\end{proof}

We used the following general monadicity result for relative tensor products.

\begin{lemma}\label{TensorMonad}
Let $\cA \rightarrow \cB$ be a functor of rigid monoidal categories. Let $\cB \curvearrowright \cM$ be a dualizable module category. Let $\act: \cM^\vee \otimes_\cA \cB \rightarrow \cM^\vee$ denote the right action. Let $\act^R: \cM \rightarrow \cB \otimes_\cA \cM$ denote the right adjoint of the left action.

Then the functor $\cM^\vee \otimes_\cA \cM \rightarrow \cM^\vee \otimes_\cB \cM$ admits a monoidal continuous right adjoint \beq\label{RightTensor} \cM^\vee \otimes_\cB \cM \rightarrow \cM^\vee \otimes_\cA \cM\eeq
that is monadic, and the resulting monad is
\[\cM^\vee \otimes_\cA \cM \xrightarrow{\id \otimes \act^R} \cM^\vee \otimes_\cA \cB \otimes_\cA \cM \xrightarrow{\act \otimes \id} \cM^\vee \otimes_\cA \cM.\]
\end{lemma}
\begin{proof}
See Corollary C.2.3 of \cite{Gai15}.
\end{proof}

\section{Essential surjectivity}
Here we will prove that $\iota^!$ is t-bounded with respect to the t-structure constructed in the previous section. We will formally deduce that it is essentially surjective on eventually coconnective subcategories.

\subsection*{t-boundedness}
Here we will prove that $\iota^!$ is t-bounded.

Recall from Proposition \ref{Monad} that the functor \[\delta_*: \IndPerf_\sfG(\widetilde{\sfG}) \underset{\IndPerf_\sfG(\sfG)}{\otimes} \IndPerf_\sfG(\widetilde{\sfG}) \rightarrow \IndPerf_\sfG(\widetilde{\sfG} \times \widetilde{\sfG})\]
is t-exact and conservative. Therefore it suffices to prove that $\iota^!$ becomes t-bounded after composing it with $\delta_*$. This has the advantage that $\widetilde{\sfG} \times \widetilde{\sfG}$ is smooth unlike $\St$, so dualizable objects in $\IndPerf_\sfG(\widetilde{\sfG} \times \widetilde{\sfG})$ are perfect.

\begin{lemma}\label{DiagonalFunctor}
Assume that $G$ has connected center. The composite \beq\label{CompositeDiagonal} \IndPerf_{\sfG}(\widetilde{\sfG}) \xrightarrow{\eqref{ABFunctor}} \cH \xrightarrow{\iota^!} \IndPerf_\sfG(\widetilde{\sfG}) \underset{\IndPerf_\sfG(\sfG)}{\otimes} \IndPerf_\sfG(\widetilde{\sfG}) \xrightarrow{\delta_*} \IndPerf_{\sfG}(\widetilde{\sfG} \times \widetilde{\sfG})\eeq
is isomorphic to direct image along the diagonal.
\end{lemma}
\begin{proof}
Proposition \ref{GrothSpringKunneth} implies that $\delta_*$ is monoidal. Theorem \ref{AB} implies that \eqref{CompositeDiagonal} is $\IndPerf_\sfG(\widetilde{\sfG})$-bilinear. Therefore it is isomorphic to direct image along the diagonal.
\end{proof}

The following lemma controls the image under $\iota^!$ of compact objects in the affine Hecke category.

\begin{lemma}\label{FiniteBounded}
Assume that $G$ has connected center. Let $A \in \cH^{\cpt}$ be compact.
\begin{enumerate}[label=(\alph*)]
\item \label{FiniteBounded1} Then $\delta_* \iota^! A \in \Perf_{\sfG}(\widetilde{\sfG} \times \widetilde{\sfG})$ is perfect.
\item \label{FiniteBounded2} The endofunctor $(\delta_* \iota^!A)*-: \IndPerf_{\sfG}(\widetilde{\sfG} \times \widetilde{\sfG}) \rightarrow \IndPerf_{\sfG}(\widetilde{\sfG} \times \widetilde{\sfG})$ is t-bounded.
\item \label{FiniteBounded3} There exists a number $n$ such that, for all $x \in W^{\fnt}$, the functors $(\delta_*\iota^!\Delta_x)*-$ and $-*(\delta_*\iota^!\nabla_x)$ have cohomological amplitude $[-n, n]$.
\end{enumerate}
\end{lemma}
\begin{proof}
Suppose that $A \in \cH^{\cpt}$ is compact, hence dualizable.
Since \[\delta_*\iota^!: \cH \rightarrow \IndPerf_{\sfG}(\widetilde{\sfG} \times \widetilde{\sfG})\]
is monoidal, $\delta_* \iota^! A$ is dualizable with respect to convolution.

Since $\widetilde{\sfG} \times \widetilde{\sfG}$ is smooth, the monoidal unit in $\IndPerf_{\sfG}(\widetilde{\sfG} \times \widetilde{\sfG})$ is compact.
Therefore all dualizable objects are compact. In particular $\delta_* \iota^! A$ is perfect and hence $\delta_* \iota^! A * -$ is t-bounded. This proves \ref{FiniteBounded1} and \ref{FiniteBounded2}.

Finally \ref{FiniteBounded3} follows from \ref{FiniteBounded2} because $W^{\fnt}$ is finite.
\end{proof}

We are now ready to prove that $\iota^!$ is t-bounded. The idea is to decompose $W = W^{\fnt} \Lambda^+ W^{\fnt}$ and observe that
\begin{enumerate}
\item[-] for dominant coweights, $\nabla_{\lambda}$ and $\Delta_{-\lambda}$ go to line bundles concentrated in degree zero,
\item[-] the remaining $W^{\fnt}$ is finite and the resulting non t-exactness is bounded by Lemma \ref{FiniteBounded}.
\end{enumerate}

\begin{proposition}\label{TBounded}
Assume that $G$ has connected center. Then the functor \[\iota^!: \cH \rightarrow \IndPerf_\sfG(\widetilde{\sfG}) \underset{\IndPerf_\sfG(\sfG)}{\otimes} \IndPerf_\sfG(\widetilde{\sfG})\]
is t-bounded with respect to the t-structure defined in Proposition \ref{Monad}.
\end{proposition}
\begin{proof}
For any $w \in W$, it is possible to write $\nabla_w = \Delta_x * \nabla_{\lambda} * \nabla_z$, where $x, z \in W^{\fnt}$ and $\lambda \in \Lambda^+$. Theorem \ref{AB}\ref{AB1} and Lemma \ref{DiagonalFunctor} imply that \beq\label{DiagonalLine} \delta_*\iota^! \nabla_{\lambda} \simeq \Delta_* \cO(\lambda).\eeq Proposition \ref{FiniteBounded}\ref{FiniteBounded3} implies that \[\delta_*\iota^! \nabla_w \simeq (\delta_* \iota^! \Delta_x) * (\Delta_*\cO(\lambda)) * (\delta_* \iota^! \nabla_z) \; \geq \; -2n.\]
Moreover $\QCoh(\sfT)$ has cohomological dimension $1 + \dim \sfT$. Hence, for any $M$ in the abelian category of $R$-modules, we get
\[\delta_*\iota^! \nabla_w \otimes_R M \; \geq \; -2n -1 - \dim \sfT.\]
Therefore $\delta_* \iota^!$ is left t-exact up to a shift by $-2n -1 - \dim \sfT$.

Similarly, it is possible to write $\Delta_w = \Delta_x * \Delta_{-\lambda} * \nabla_z$, where $x, z \in W^{\fnt}$ and $\lambda \in \Lambda^+$. Taking the monoidal inverse of \eqref{DiagonalLine} gives $\delta_*\iota^! \Delta_{-\lambda} \simeq \Delta_* \cO(-\lambda)$. Therefore a similar argument shows that $\delta_* \iota^!$ is right t-exact up to a shift by $2n$.

We have now proved that $\delta_* \iota^!$ is t-bounded. Recall from Proposition \ref{Monad} that $\delta_*$ is t-exact and conservative. Therefore $\iota^!$ is also t-bounded as desired.
\end{proof}

\subsection*{Essential surjectivity}
The above t-boundedness formally implies that $\iota^!$ is essentially surjective on eventually coconnective subcategories by the following general result.

\begin{proposition}\label{GeneralEssentialSurjectivity}
Let $\iota^!: \cC \rightarrow \cD$ be a continuous functor that admits a fully faithful left adjoint $\iota_!$. Assume that $\cC$ and $\cD$ admit t-structures such that $\iota^!$ is t-bounded.

Then on eventually coconnective subcategories \beq\label{CoconnectiveFunctor} \iota^!: \cC^{> -\infty} \rightarrow \cD^{> -\infty}\eeq
is essentially surjective.
\end{proposition}
\begin{proof}
Let $\cE \in \cD^{> -\infty}$ be eventually coconnective, and we will show that it is hit by $\cC^{> -\infty}$.

Since $\iota^!$ is continuous, it admits a (possibly discontinuous) right adjoint $\iota_\#$. Because $\iota^!$ is right t-exact up to a shift, $\iota_\#$ is left t-exact up to a shift. Therefore $\iota_\# \cE > -\infty$ is also eventually coconnective.

Since the left adjoint $\iota_!$ is fully faithful, Lemma \ref{AdjointTriple} implies that the right adjoint $\iota_\#$ is also fully faithful. Hence $\cE \simeq \iota^! \iota_\# \cE$ belongs to the essential image of \eqref{CoconnectiveFunctor}.
\end{proof}

To show that the right adjoint was fully faithful, we used the following general categorical observation about adjoint triples.

\begin{lemma}\label{AdjointTriple}
Let $\iota^!$ be a functor that admits both a left adjoint $\iota_!$ and a right adjoint $\iota_\#$. Then $\iota_!$ is fully faithful if and only if $\iota_\#$ is fully faithful.
\end{lemma}
\begin{proof}
Note that $\iota^! \iota_!$ is left adjoint to $\iota^! \iota_\#$. Therefore $\iota^!\iota_! \simeq \id$ if and only if $\iota^!\iota_\# \simeq \id$.
\end{proof}

\section{Proof of the equivalence}
Having already proved that $\iota^!$ induces a t-bounded equivalence on eventually coconnective subcategories, it remains to show that compact objects in the affine Hecke category match coherent sheaves on the Steinberg stack. For this we will characterize both as the `pseudo-compact' objects with respect to the t-structures.

\subsection*{Pseudo-compactness}
Here we characterize the pseudo-compact objects on both the coherent and constructible sides.

\begin{definition}
Let $\cC$ be a DG-category equipped with a t-structure that is compatible with filtered colimits. We will call an object $A \in \cC$ `pseudo-compact' if
\begin{enumerate}[label=(\roman*)]
\item\label{Psc1} $A > -\infty$ is eventually coconnective,
\item\label{Psc2} $\Hom(A, -)$ commutes with uniformly bounded below filtered colimits.
\end{enumerate}
Write $\cC^{\psc}$ for the full subcategory of pseudo-compact objects in $\cC$.
\end{definition}

For quasi-coherent sheaves, the following proposition says that pseudo-compactness is equivalent to bounded coherence.

\begin{lemma}\label{QCohPsc}
The subcategories $\QCoh_{\sfG}(\St)^{\psc} \simeq \Coh_{\sfG}(\St)$ coincide.
\end{lemma}
\begin{proof}
Lemma 9.21 of \cite{Zhu25} says that the functor from ind-coherent sheaves to quasi-coherent sheaves induces an equivalence on eventually coconnective subcategories
\beq\label{ZhuBelow}\IndCoh_\sfG(\St)^{> -\infty} \simeq \QCoh_\sfG(\St)^{> -\infty}.\eeq
Therefore all coherent sheaves in $\QCoh_\sfG(\St)$ are pseudo-compact.

Conversely suppose that $\cE \in \QCoh_\sfG(\St)$ is pseudo-compact. In particular \ref{Psc1} gives $\cE \geq -n$ for some $n$. Thus \eqref{ZhuBelow} implies that $\cE \simeq \colim \cF_i$ is a filtered colimit of coherent sheaves $\cF_i \in \Coh_\sfG(\St)$. Applying the truncation functor $\tau^{\geq -n}$ gives $\cE \simeq \colim \tau^{\geq -n} \cF_i$. Now \ref{Psc2} implies that $\cE$ is a retract of one of the terms $\tau^{\geq -n} \cF_i$ in the bounded below filtered colimit. Since coherent sheaves are closed under truncation, $\cE$ is coherent.
\end{proof}

For the affine Hecke category, the following proposition says that pseudo-compactness is equivalent to compactness.

\begin{lemma}\label{HPsc}
The subcategories $\cH^{\psc} \simeq \cH^{\cpt}$ coincide.
\end{lemma}
\begin{proof}
Suppose that $A \in \cH^{\psc}$ is pseudo-compact. Then $A \geq -n$ for some $n$. Recall that $A$ is the filtered colimit of its $!$-restrictions to finite dimensional strata. These $!$-restrictions are all concentrated in degree $\geq -n$. Therefore pseudo-compactness gives that $A$ has finite dimensional support.

Therefore $A$ admits a finite Cousin filtration whose graded pieces are of the form $\Delta_w \otimes_{\cO_\sfT} \cE_w$ for some $\cE_w \in \QCoh(\sfT)$. For each $w \in W$, homing into bounded below filtered colimits $*$-extended from $\Fl_w$ gives that $\cE_w$ is pseudo-compact. Since $\sfT$ is smooth, each $\cE_w$ is compact. Therefore $A$ is compact. 

Conversely if $A \in \cH^{\cpt}$ is compact, then it is a finite colimit of standards, hence eventually coconnective.
\end{proof}

\subsection*{Proof of Theorem \ref{Main}}
We are now ready to prove the tame local Betti Langlands equivalence.

First assume that $G$ has connected center. Then Propositions \ref{TBounded} and \ref{GeneralEssentialSurjectivity} give a t-bounded essentially surjective functor \[\iota^!: \cH^{> -\infty} \rightarrow (\IndPerf_\sfG(\widetilde{\sfG}) \underset{\IndPerf_\sfG(\sfG)}{\otimes} \IndPerf_\sfG(\widetilde{\sfG}))^{> -\infty}.\]
Propositions \ref{BoundedCompact} and \ref{GeneralFullyFaithful} imply that it is fully faithful. Moreover Proposition \ref{Monad}\ref{Monad2} identifies the target with eventually coconnective quasi-coherent sheaves on the Steinberg stack. Therefore we get a t-bounded equivalence \[\cH^{> -\infty} \simeq \QCoh_\sfG(\St)^{> -\infty}.\]

By Lemmas \ref{QCohPsc} and \ref{HPsc}, passing to pseudo-compact objects gives an equivalence $\cH^{\cpt} \simeq \Coh_{\sfG}(\St)$. Finally ind-completing gives the desired monoidal equivalence $\cH \simeq \IndCoh_\sfG(\St)$.

Now let $G$ be an arbitrary reductive group. Write $\sfG = \sfG^{\sct}/\sfZ$, where $\sfG^{\sct}$ has simply connected derived subgroup and $\sfZ$ is a finite central subgroup. Let $\St^{\sct}$ be the Steinberg scheme for $\sfG^{\sct}$. Write $\cH^{\ad}$ for the affine Hecke category for group that is Langlands dual to $\sfG^{\sct}$.

We have already proved an equivalence \[\cH^{\ad} \simeq \IndCoh_{\sfG^{\sct}}(\St^{\sct}).\]
Moreover this equivalence is $\sfZ$-equivariant by \cite{DT25}. Passing to $\sfZ$-equivariant objects and restricting to the component indexed by the trivial character of $\sfZ$ gives the desired monoidal equivalence $\cH \simeq \IndCoh_\sfG(\St)$.

\end{document}